\documentclass[12pt]{article}
\usepackage{amsmath}
\usepackage{amssymb}
\usepackage{amsthm}
\usepackage[pdftex]{graphicx}
\usepackage{array}
\usepackage{tikz}
\usepackage{url}

\usepackage{fancyhdr}
\newtheorem{thm}{Theorem}[section]
\newtheorem{lem}[thm]{Lemma}
\newtheorem{prop}[thm]{Proposition}
\newtheorem{conj}{Conjecture}[section]

\newtheorem{defn}{Definition}[section]
\theoremstyle{definition}

\theoremstyle{remark}
\newtheorem*{rem}{Remark}

\newcommand{\real}{\mathbb{R}}

\newcommand{\cplx}{\mathbb{C}}
\newcommand{\nats}{\mathbb{N}}

\newcommand{\al}{\alpha}

\newcommand{\eps}{\epsilon}

\renewcommand{\phi}{\varphi}
\newcommand{\gam}{\gamma}
\newcommand{\del}{\delta}
\newcommand{\sg}{\sigma}
\newcommand{\kap}{\kappa}

\newcommand{\Lam}{\Lambda}
\newcommand{\Del}{\Delta}
\newcommand{\Gam}{\Gamma}

\renewcommand{\bar}{\overline}

\DeclareFontFamily{OT1}{rsfs}{}
\DeclareFontShape{OT1}{rsfs}{n}{it}{<-> rsfs10}{}
\DeclareMathAlphabet{\mathscr}{OT1}{rsfs}{n}{it}

\begin{document}

\begin{center}
\LARGE
Bounding sums of the M\"obius function over arithmetic progressions

\vspace{0.2in}
\normalsize
Lynnelle Ye

\end{center}

\begin{abstract}
Let $M(x)=\sum_{1\le n\le x}\mu(n)$ where $\mu$ is the M\"obius function. It is well-known that the Riemann Hypothesis is equivalent to the assertion that $M(x)=O(x^{1/2+\epsilon})$ for all $\epsilon>0$. There has been much interest and progress in further bounding $M(x)$ under the assumption of the Riemann Hypothesis. In 2009, Soundararajan established the current best bound of
\[
M(x)\ll\sqrt{x}\exp\left((\log x)^{1/2}(\log\log x)^c\right)
\]
(setting $c$ to $14$, though this can be reduced). Halupczok and Suger recently applied Soundararajan's method to bound more general sums of the M\"obius function over arithmetic progressions, of the form 
\[
M(x;q,a)=\sum_{\substack{n\le x \\ n\equiv a\pmod{q}}}\mu(n).
\]
They were able to show that assuming the Generalized Riemann Hypothesis, $M(x;q,a)$ satisfies
\[
M(x;q,a)\ll_{\epsilon}\sqrt{x}\exp\left((\log x)^{3/5}(\log\log x)^{16/5+\epsilon}\right)
\]
for all $q\le\exp\left(\frac{\log 2}2\lfloor(\log x)^{3/5}(\log\log x)^{11/5}\rfloor\right)$, with $a$ such that $(a,q)=1$, and $\epsilon>0$.

In this paper, we improve Halupczok and Suger's work to obtain the same bound for $M(x;q,a)$ as Soundararajan's bound for $M(x)$ (with a $1/2$ in the exponent of $\log x$), with no size or divisibility restriction on the modulus $q$ and residue $a$.
\end{abstract}

\section{Introduction}

The Riemann Hypothesis (RH) is a conjecture about the zeros of the Riemann zeta function, a meromorphic function $\zeta:\cplx\to\cplx$ defined as follows: for $s\in\cplx$ with $\Re(s)>1$, we let 
\[
\zeta(s)=\sum_{n=1}^{\infty}\frac1{n^s}.
\] 
For $\Re(s)\le1$, $\zeta(s)$ is the value of the unique analytic continuation of this function to the rest of the complex plane. The function $\zeta$ is holomorphic everywhere except at $s=1$, where it has a simple pole. For $\Re(s)>1$, we can also write $\zeta(s)$ in its Euler product representation
\[
\zeta(s)=\prod_{p\text{ prime}}\left(1-\frac1{p^s}\right)^{-1}.
\]

The Riemann zeta function satisfies an important functional equation relating it to the $\Gam$ function, which is defined by $\Gam(s)=\int_0^{\infty}x^{s-1}e^{-x}dx$ for $s\in\cplx$ such that $\Re(s)>0$ and analytically continued to the rest of the complex plane. The functional equation is
\[
\pi^{-s/2}\Gam\left(\frac s2\right)\zeta(s)=\pi^{-(1-s)/2}\Gam\left(\frac{1-s}2\right)\zeta(1-s).
\]

The $\Gam$ function is nonzero everywhere, and it has simple poles exactly at the nonpositive integers. From this and the above equation, we can see that $\zeta(s)=0$ when $s$ is a negative even integer. From the Euler product of $\zeta$, we may compute for $\Re(s)>1$ that
\[
\frac1{\zeta(s)}=\sum_{n=1}^{\infty}\frac{\mu(n)}{n^s}
\]
where $\mu(n)$ is the multiplicative function, called the M\"obius function, defined by 
\[
 \mu(n) =
  \begin{cases}
   0 & \text{if } p^2|n \text{ for some prime } p, \\
   1       & \text{if } n = p_1\dotsb p_{2k} \text{ for distinct primes } p_i, \\
   -1 & \text{if } n=p_1\dotsb p_{2k-1} \text{ for distinct primes } p_i.
  \end{cases}
\]
Since $\sum_{n=1}^{\infty}\frac{\mu(n)}{n^s}$ is absolutely convergent for $\Re(s)>1$, we know $\zeta(s)\neq0$ for $\Re(s)>1$, and hence from the functional equation that $\zeta(s)\neq0$ for $\Re(s)<0$, except for the previously mentioned negative even integers. So all other zeros of $\zeta$ lie in the strip $0\le\Re(s)\le1$ (referred to as the critical strip) and are known as the ``nontrivial zeros''. 

The behavior of the zeros of the Riemann zeta function is deeply connected to the behavior of the primes in $\nats$. Let $\pi(x)$ be the number of primes no greater than $x$; then the Prime Number Theorem, which states that $\lim_{x\to\infty}\frac{\pi(x)}{x/\log x}=1$, follows from the fact that $\zeta$ has no zeros on the line $\Re(s)=1$. More specific information about the distribution of primes follows from more specific information about the zeros of $\zeta$, such as would be given by the Riemann Hypothesis.

\begin{conj}[Riemann Hypothesis]
All nontrivial zeros of $\zeta(s)$ lie on the line $\Re(s)=\frac12$, called the critical line.
\end{conj}

From the fact that $\zeta$ has no zeros on the line $\Re(s)=1$, the more specific version of the Prime Number Theorem as proven by Hadamard and de la Vall\'ee Poussin is that 
\[
\pi(x)=\int_2^x\frac1{\log t}dt+O(x\exp(-a\sqrt{\log x}))
\]
for some $a>0$. Assuming the Riemann Hypothesis, this would be improved to
\[
\pi(x)=\int_2^x\frac1{\log t}dt+O(\sqrt{x}\log x).
\]

The Mertens function $M(x)$ is defined by $M(x)=\sum_{1\le n\le x}\mu(n)$ where $\mu$ is the M\"obius function. It is well-known that the Riemann Hypothesis is equivalent to the assertion that $M(x)=O(x^{1/2+\eps})$ for all $\eps>0$ (see~\cite{iwakow}, for example, for details). There has been much interest and progress in further bounding $M(x)$ under the assumption of the Riemann Hypothesis. Landau (\cite{landau}) showed in 1924 that RH implies $M(x)=O(x^{1/2+\eps})$ with $\eps\ll \log\log\log x/\log\log x$, and Titchmarsh (\cite{titchmarsh}) in 1927 improved this to $\eps\ll 1/\log\log x$. In 2009, Maier and Montgomery (\cite{maimont}) gave a much improved bound of
\[
M(x)\ll\sqrt{x}\exp\left(C(\log x)^{39/61}\right).
\]
Finally, in~\cite{sound}, Soundararajan established the current best bound of
\[
M(x)\ll\sqrt{x}\exp\left((\log x)^{1/2}(\log\log x)^c\right)
\]
(where Soundararajan set $c$ to $14$, though Balazard and de Roton decrease it to $5/2+\eps$ in~\cite{baladero}). Gonek has conjectured (see Ng in~\cite{ng} for details) that the true range of $|M(x)|$ goes up precisely to the much smaller $O\left(\sqrt{x}(\log\log\log x)^{5/4}\right)$.

Halupczok and Suger recently applied Soundararajan's method to bound more general sums of the M\"obius function over arithmetic progressions, of the form 
\[
M(x;q,a)=\sum_{\substack{n\le x \\ n\equiv a\pmod{q}}}\mu(n).
\]
This requires a generalized version of the Riemann Hypothesis, appropriately called the Generalized Riemann Hypothesis (GRH). The GRH deals with Dirichlet $L$-functions, defined by
\[
L(s,\chi)=\sum_{n=1}^{\infty}\frac{\chi(n)}{n^s}
\]
and analytically continued to the complex plane, where $\chi$ is a certain multiplicative function called a Dirichlet character. There are $\phi(q)$ distinct characters $\chi$ for each modulus $q$, each with range contained in the $q$th roots of unity, satisfying
\[
\frac1{\phi(q)}\sum_{\chi}\chi(n)\bar{\chi}(a)=
\begin{cases}
1 & \text{ if } n\equiv a\pmod{q} \\
0 & \text{ otherwise.}
\end{cases}
\]
Like $\zeta(s)$, $L(s,\chi)$ has a functional equation relating it to $\Gam$, which looks like
\[
\left(\frac q{\pi}\right)^{s/2}\Gam\left(\frac{\kap+s}2\right)L(s,\chi)
=\eps(\chi)\left(\frac q{\pi}\right)^{(1-s)/2}\Gam\left(\frac{\kap+1-s}2\right)L(1-s,\bar{\chi})
\]
where $\kap$ is equal to $0$ if $\chi(-1)=1$ and $1$ if $\chi(-1)=-1$. Consequently its zeros also have a similar structure.

\begin{conj}[Generalized Riemann Hypothesis]
For every $\chi$, all nontrivial zeros of $L(s,\chi)$ lie on the line $\Re(s)=\frac12$.
\end{conj}
As with the consequences of RH for the distribution of primes, GRH would imply, among other things, similar consequences for the distribution of primes in arithmetic progressions.

Halupczok and Suger showed that assuming GRH, given $\eps>0$, the bound
\[
M(x;q,a)\ll_{\eps}\sqrt{x}\exp\left((\log x)^{3/5}(\log\log x)^{16/5+\eps}\right)
\]
holds for all progressions $a\pmod{q}$ with $q\le\exp\left(\frac{\log 2}2\lfloor(\log x)^{3/5}(\log\log x)^{11/5}\rfloor\right)$ and $(a,q)=1$.

In this paper, we improve Halupczok and Suger's work to obtain a bound of strength equal to that of Soundararajan, with no restriction on the modulus $q$ or the residue $a$. Our main theorem is as follows.
\begin{thm}
\label{main}
Assuming GRH, given $\eps>0$, the bound
\[
M(x;q,a)\ll_{\eps}\sqrt{x/d}\exp\left((\log(x/d))^{1/2}(\log\log(x/d))^{3+\eps}\right)
\]
holds uniformly for all progressions $a\pmod{q}$, where $d=\gcd(a,q)$.
\end{thm}

Since $|\mu(n)|\le1$, the trivial bound on $M(x;q,a)$ is $x/q$. Our bound therefore remains nontrivial whenever $q\le x^{1/2-\eps}$.

In Section~\ref{prelim}, we provide preliminaries on the explicit formula for sums over zeros of $L$-functions, consequent bounds on the deviation from average of the number of zeros in an interval of the critical line, and the large sieve. In Section~\ref{pointcount}, we bound the number of well-separated ordinates on the critical line with an unusual accumulation of zeros of $L(s,\chi)$ nearby. In 
Section~\ref{Lfuncbds}, we give bounds for $L(s,\chi)$ depending on the number of zeros near $s$. In Section~\ref{finalthm}, we prove Theorem~\ref{main}.

We assume GRH for the rest of this paper.

\section*{Acknowledgements}

I would like to thank Kannan Soundararajan for advising me on this project, and the Stanford Undergraduate Research Institute in Mathematics for its support.

\section{Preliminaries}
\label{prelim}

First consider the case $(a,q)=1$. Using the Dirichlet characters $\chi\pmod{q}$ and Perron's formula, we may write
\begin{align*}
M(x;q,a) &=\frac1{\phi(q)}\sum_{\chi\pmod{q}}\bar{\chi}(a)\sum_{n\le x}\chi(n)\mu(n) \\
&=\frac1{\phi(q)}\sum_{\chi\pmod{q}}\frac{\bar{\chi}(a)}{2\pi i}\int_{1+1/\log x-i\lfloor x\rfloor}^{1+1/\log x+i\lfloor x\rfloor}\frac{x^s}{sL(s,\chi)}ds+O(\log x).
\end{align*}
More generally, let $a=bd,q=rd$ with $(b,r)=1$. We have
\[
M(x;q,a) = \sum_{\substack{n\le x\\ n\equiv bd\pmod{rd}}}\mu(n) = \sum_{\substack{m\le x/d\\ m\equiv b\pmod{r}}} \mu(dm)=\mu(d)\sum_{\substack{m\le x/d\\ m\equiv b\pmod{r}\\ (d,m)=1}} \mu(m)
\]
so the sum we are interested in is really just
\[
\sum_{\substack{m\le x/d\\ m\equiv b\pmod{r}\\ (d,m)=1}} \mu(m)=\frac1{\phi(r)}\sum_{\chi\pmod{r}}\bar{\chi}(b)\sum_{\substack{m\le x/d\\ (d,m)=1}}\chi(m)\mu(m)
\]
\[
=\frac1{\phi(r)}\sum_{\chi\pmod{r}}\frac{\bar{\chi}(b)}{2\pi i}\int_{1+1/\log(x/d)-i\lfloor x/d\rfloor}^{1+1/\log(x/d)+i\lfloor x/d\rfloor}\frac{(x/d)^s}{sL(s,\chi)l_d(s,\chi)}ds+O(\log(x/d))
\]
where
\[
l_d(s,\chi)=\prod_{p|d}\left(1-\frac{\chi(p)}{p^s}\right).
\]
We are now concerned with bounding the average value of the integrand $\frac{(x/d)^s}{sL(s,\chi)l_d(s,\chi)}$ over all $\chi$, as this suffices to bound $M(x;q,a)$. We first require the following lemmas. The first constructs smooth approximations to the characteristic function of an interval.

\begin{lem}
\label{charfnapprox}
Let $\chi_{[-h,h]}$ be the characteristic function of the interval $[-h,h]$. There are even entire functions $F_+$ and $F_-$ depending on $h$ and $\Delta$ which satisfy the following properties:

1. $F_-(u)\le\chi_{[-h,h]}(u)\le F_+(u)$ for $u\in\real$.

2. $\int_{-\infty}^{\infty}|F_{\pm}(u)-\chi_{[-h,h]}(u)|du\le\frac1{\Del}$.

3. $\hat{F}_{\pm}(x)=0$ for $|x|\ge\Del$, and $\hat{F}_{\pm}(x)=\frac{\sin(2\pi hx)}{\pi x}+O(1/\Del)$.

4. If $z=x+iy$ with $|z|\ge2h$ then $|F_{\pm}(z)|\ll\frac{e^{2\pi\Del|y|}}{(\Del|z|)^2}$.
\end{lem}

\begin{proof}
Such functions were originally constructed by Selberg (\cite{selberg}) using Beurling's approximation to the sign function. They are discussed in detail in~\cite{vaaler}. We will just state the relevant formulas. Let
\[
K(z)=\frac{(\sin \pi z)^2}{(\pi z)^2},\hspace{0.2in} H(z)=\left(\frac{\sin\pi z}{\pi}\right)^2\left(\sum_{n=-\infty}^{\infty}\frac{\text{sgn}(n)}{(z-n)^2}+\frac2z\right),
\]
where $\text{sgn}(x)$ is the sign function, with values $1$ for $x>0$, $-1$ for $x<0$, and $0$ for $x=0$. Then the functions $F_{\pm}$ are
\[
F_{\pm}(z)=\frac12(H(\Del(z+h))\pm K(\Del(z+h))+H(\Del(h-z))\pm K(\Del(h-z))).
\]
\end{proof}

Next we need the following form of the explicit formula for sums over zeros of $L$-functions.

\begin{lem}
\label{explicitformula}
Let $\chi$ be a primitive character $\pmod{q}$. Let $f(s)$ be analytic in the strip $|\Im(s)|\le1/2+\eps$ for some $\eps>0$, real-valued on the real line, and such that $|f(s)|\ll(1+|s|)^{-1-\del}$ for some $\del>0$. Then if $\rho=1/2+i\gam$ runs over the nontrivial zeroes of $L(s,\chi)$, we have
\begin{align*}
\sum_{\rho}f(\gam)=&\frac1{2\pi}\hat{f}(0)\log\frac{q}{\pi}+(1-\kap)f\left(\frac1{2i}\right)
+\frac1{2\pi}\int_{-\infty}^{\infty}f(u)\Re\frac{\Gam'}{\Gam}\left(\frac{1/2+\kappa+iu}2\right)du \\
&-\frac1{2\pi}\sum_{n=2}^{\infty}\frac{\Lam(n)\chi(n)}{\sqrt{n}}\left(\hat{f}\left(\frac{\log n}{2\pi}\right)+\hat{f}\left(-\frac{\log n}{2\pi}\right)\right)
\end{align*}
where $\kap$ is equal to $0$ if $\chi(-1)=1$ and $1$ if $\chi(-1)=-1$.
\end{lem}

\begin{proof}
This is a specialization of Theorem 5.12 in~\cite{iwakow}. We reproduce the proof here.

We begin by proving a version of this identity for Mellin transforms. This states that for a $C^{\infty}$ function $\phi:(0,+\infty)\to\cplx$ with compact support and Mellin transform $\tilde{\phi}(s)=\int_0^{\infty}\phi(x)x^{s-1}dx$, and setting $\psi(x)=x^{-1}\phi(x^{-1})$, we have
\begin{align*}
\sum_{\rho}\tilde{\phi}(\rho)= &\phi(1)\log\frac{q}{\pi}+(1-\kap)\tilde{\phi}(1)+\frac1{2\pi}\int_{-\infty}^{\infty}\tilde{\phi}(u)\Re\frac{\Gam'}{\Gam}\left(\frac{1/2+\kappa+iu}2\right)du\\
&-\sum_n(\Lam(n)\chi(n)\phi(n)+\Lam(n)\bar{\chi}(n)\psi(n)).
\end{align*}
To obtain this, we recall that the functional equation for $L(s,\chi)$ takes the form
\[
\left(\frac q{\pi}\right)^{s/2}\Gam\left(\frac{\kap+s}2\right)L(s,\chi)
=\eps(\chi)\left(\frac q{\pi}\right)^{(1-s)/2}\Gam\left(\frac{\kap+1-s}2\right)L(1-s,\bar{\chi})
\]
from which
\[
\log\frac{q}{\pi}+\frac12\frac{\Gam'}{\Gam}\left(\frac{\kap+s}2\right)+\frac12\frac{\Gam'}{\Gam}\left(\frac{\kap+1-s}2\right)+\frac{L'}{L}(s,\chi)+\frac{L'}{L}(1-s,\bar{\chi})=0.
\]
We multiply this equation by $\tilde{\phi}(s)$, integrate along the line $\Re(s)=2-\del$, and divide by $2\pi i$. The first term becomes $\phi(1)\log q$ by Mellin inversion. The line of integration for the second and third terms (with $\Gam'/\Gam$) may be shifted to $\Re(s)=1/2$ while picking up a single pole at $s=1$ if $\kap=0$ and nothing otherwise (giving the term $(1-\kap)\tilde{\phi}(1)$), after which the two terms in the integrand are conjugates and cancel to give $\tilde{\phi}(u)\Re\frac{\Gam'}{\Gam}\left(\frac{1/2+\kappa+iu}2\right)$ for $s=1/2+iu$. We shift the line of integration for the fourth and fifth terms (with $L'/L$) to $\Re(s)=-\del$, picking up a pole with residue $\tilde{\phi}(\rho)$ for each zero $\rho$ of $L$ on the $1/2$-line. Finally, the remaining integrals may be written as $\sum_n(\Lam(n)\chi(n)\phi(n)+\Lam(n)\bar{\chi}(n)\psi(n))$ by Mellin inversion.

The Fourier transform version follows from setting $\phi(x)=\frac{1}{2\pi}x^{-1/2}\hat{f}\left(\frac{\log x}{2\pi}\right)$, $x=e^{2\pi y}$.
\end{proof}

Let $\chi$ be a primitive character $\pmod{q}$ and let $N(t,\chi)$ be the number of zeros $1/2+i\gam$ of $L(s,\chi)$ lying in the interval $0\le\gam\le t$ on the critical line, so that the ``average'' value of $N(t+h,\chi)-N(t-h,\chi)$ is $\frac h{\pi}\log\frac{qt}{2\pi}$. We now apply the explicit formula to our characteristic function approximations $F_{\pm}$ to estimate the deviation of the actual value of $N(t+h,\chi)-N(t-h,\chi)$ from the average as a Dirichlet polynomial.

\begin{lem}
\label{zerostoprimesum}
Let $t\ge25$, $\Del\ge2$, and $0<h\le\sqrt{t}$. Let $\chi$ be a primitive character $\pmod{q}$. Then
\[
N(t+h,\chi)-N(t-h,\chi)-\frac h{\pi}\log\frac{qt}{2\pi}\le\frac{\log(qt)}{2\pi\Del}
-\frac1{\pi}\Re\left(\sum_{p\le e^{2\pi\Del}}\frac{\chi(p)\log p}{p^{1/2+it}}\hat{F}_+\left(\frac{\log p}{2\pi}\right)\right)+O(\log\Del)
\]
and
\[
N(t+h,\chi)-N(t-h,\chi)-\frac h{\pi}\log\frac{qt}{2\pi}\ge-\frac{\log(qt)}{2\pi\Del}
-\frac1{\pi}\Re\left(\sum_{p\le e^{2\pi\Del}}\frac{\chi(p)\log p}{p^{1/2+it}}\hat{F}_-\left(\frac{\log p}{2\pi}\right)\right)+O(\log\Del)
\]
where $F_{\pm}$ are the functions from Lemma~\ref{charfnapprox}.
\end{lem}

\begin{proof}
Let $f(s)=F_{\pm}(s-t)$, so that $\hat{f}(x)=\hat{F}_{\pm}(x)e^{-2\pi ixt}$. We have
\[
\sum_{\rho}F_-(\gam-t)\le N(t+h,\chi)-N(t-h,\chi)=\sum_{\rho}\chi_{[-h,h]}(\gam-t)\le\sum_{\rho}F_+(\gam-t).
\]
Since $F_{\pm}$ is real and even on the real line, $\hat{F}_{\pm}$ is also real and even. Hence
\[
\hat{f}\left(\frac{\log n}{2\pi}\right)+\hat{f}\left(-\frac{\log n}{2\pi}\right)=n^{-it}\hat{F}_{\pm}\left(\frac{\log n}{2\pi}\right)+n^{it}\hat{F}_{\pm}\left(-\frac{\log n}{2\pi}\right)=\Re\left(\frac{1}{n^{it}}\hat{F}_{\pm}\left(\frac{\log n}{2\pi}\right)\right)
\]
so that using the explicit formula from Lemma~\ref{explicitformula}, we find
\begin{align*}
\sum_{\rho}F_{\pm}(\gam-t)= &\frac1{2\pi}\hat{F}_{\pm}(0)\log\frac{q}{\pi}+(1-\kap)F_{\pm}\left(\frac1{2i}-t\right)-\frac1{\pi}\Re\left(\sum_{n=2}^{\infty}\frac{\Lam(n)\chi(n)}{n^{1/2+it}}\hat{F}_{\pm}\left(\frac{\log n}{2\pi}\right)\right) \\
&+\frac1{2\pi}\int_{-\infty}^{\infty}F_{\pm}(u-t)\Re\frac{\Gam'}{\Gam}\left(\frac{1/2+\kap+iu}2\right)du.
\end{align*}
We have $F_{\pm}\left(\frac1{2i}-t\right)\ll\frac1{\Del^2t^2}$ from Property 4 of Lemma~\ref{charfnapprox}. Also, as shown in~\cite{goldgon}, we have that
\[
\frac1{2\pi}\int_{-\infty}^{\infty}F_{\pm}(u-t)\Re\frac{\Gam'}{\Gam}\left(\frac{1/2+\kap+iu}2\right)du=\frac1{2\pi}\log\frac t2\hat{F}_{\pm}(0)+O(1).
\]
We reproduce the proof here, using Stirling's approximation for $\Gam'/\Gam$. From Property 4 of Lemma~\ref{charfnapprox}, we have
\[
\int_{t+4\sqrt{t}}^{\infty}F_{\pm}(u-t)\Re\frac{\Gam'}{\Gam}\left(\frac{1/2+\kap+iu}2\right)du\ll\int_{t+4\sqrt{t}}^{\infty}\frac{\log(u+2)}{\Del^2(u-t)^2}du
\ll\frac{\log t}{\sqrt{t}}
\]
and similarly for the integral from $-\infty$ to $t-4\sqrt t$, while
\begin{align*}
\int_{t-4\sqrt t}^{t+4\sqrt t}F_{\pm}(u-t)\Re\frac{\Gam'}{\Gam}\left(\frac{1/2+\kap+iu}2\right)du 
&= \int_{t-4\sqrt t}^{t+4\sqrt t}F_{\pm}(u-t)\left(\log\frac u2 + O(1/u)\right)du \\
&= \int_{t-4\sqrt t}^{t+4\sqrt t}F_{\pm}(u-t)\left(\log\frac t2 + O(1/t)\right)du \\
&= \int_{-\infty}^{\infty}F_{\pm}(u-t)\log\frac t2du+O\left(\frac{\log(t/2)+2h}{t}\right)\\
&= \left(\log\frac t2\right)\hat{F}(0)+O(1).
\end{align*}
Finally, since $\hat{F}_{\pm}\left(\frac{\log n}{2\pi}\right)=0$ for $n\ge e^{2\pi\Del}$, we have
\[
\Re\left(\sum_{n=2}^{\infty}\frac{\Lam(n)\chi(n)}{n^{1/2+it}}\hat{F}_{\pm}\left(\frac{\log n}{2\pi}\right)\right)
=\Re\left(\sum_{n\le e^{2\pi\Del}}\frac{\Lam(n)\chi(n)}{n^{1/2+it}}\hat{F}_{\pm}\left(\frac{\log n}{2\pi}\right)\right)
\]
\[
=\Re\left(\sum_{p\le e^{2\pi\Del}}\frac{\Lam(p)\chi(p)}{p^{1/2+it}}\hat{F}_{\pm}\left(\frac{\log p}{2\pi}\right)\right)
+\Re\left(\sum_{p\le e^{\pi\Del}}\frac{\Lam(p)\chi(p)}{p^{1+2it}}\hat{F}_{\pm}\left(\frac{\log p}{\pi}\right)\right)+O(1)
\]
\[
=\Re\left(\sum_{p\le e^{2\pi\Del}}\frac{\Lam(p)\chi(p)}{p^{1/2+it}}\hat{F}_{\pm}\left(\frac{\log p}{2\pi}\right)\right)+O(\log\Del).
\]
\end{proof}

Finally, we need the following version of the large sieve inequality.

\begin{defn}
A set of points $\{s_r^{\chi}=\sg_r^{\chi}+it_r^{\chi}\}$, where $r=1,\dotsc,R_{\chi}$ and $\chi$ ranges over the characters $\pmod{q}$, is well-separated if, for any given $\chi$, we have $|t_i^{\chi}-t_j^{\chi}|\ge1$ for all $i\neq j$.
\end{defn}

\begin{prop}
\label{sieve}
Let $A(s,\chi)=\sum_{p\le N}a(p)\chi(p)p^{-s}$ be a twisted Dirichlet polynomial. Let $T$ be large and let $\{s_r^{\chi}\}$ be a set of well-separated points with $T<t_1^{\chi}<t_2^{\chi}<\dotsb<t_{R_{\chi}}^{\chi}\le 2T$ for all $\chi$ and $\sg_r^{\chi}\ge\al$. Then for any $k$ with $N^k\le qT$, we have
\[
\sum_{\chi}\sum_{r=1}^{R_{\chi}}|A(s_r^{\chi},\chi)|^{2k}\ll \phi(q)T(\log T)^2k!\left(\sum_{p\le N}|a(p)|^2p^{-2\al}\right)^k.
\]
\end{prop}

\begin{proof}
The proof follows that of Proposition 4 in~\cite{maimont}; it is identical except for the inclusion of a summation over characters. Let 
\[
B(s,\chi)=A(s,\chi)^k=\sum_{n\le N^k}c_n\chi(n)n^{-s}
\] 
and let $D_r(s)$ denote the disk of radius $r$ centered at $s$. By the mean-value property of harmonic functions and the Cauchy-Schwarz inequality, we have
\[
|B(s,\chi)|^2\le\frac{(\log T)^2}{\pi}\iint_{D_{1/\log T}(s)}|B(x+iy,\chi)|^2dxdy
\]
for all $s$. Since the disks $D_{1/\log T}(s_r^{\chi})$ are disjoint and lie in the half-strip $\sg\ge\al-1/\log T$, between the lines $t=T-1$ and $t=2T+1$, we can write
\[
\sum_{\chi}\sum_{r=1}^{R_{\chi}}|B(s_r^{\chi},\chi)|^2\ll(\log T)^2\int_{\al-1/\log T}^{\infty}\sum_{\chi}\int_{T-1}^{2T+1}|B(\sg+it,\chi)|^2dtd\sg.
\]
By the proof of Theorem 6.4 of~\cite{mont}, specifically Equation 6.14, we have
\[
\sum_{\chi}\int_{T-1}^{2T+1}|B(\sg+it,\chi)|^2dt=\phi(q)(T+O(N^k/q))\sum_n\frac{|c_n|^2}{n^{2\sg}}
\]
which, using the condition that $N^k\le qT$, plugging back into the previous inequality, and integrating with respect to $\sg$, gives
\[
\sum_{\chi}\sum_{r=1}^{R_{\chi}}|B(s_r^{\chi},\chi)|^2\ll\phi(q)T(\log T)^2\sum_n\frac{|c_n|^2}{n^{2\al}\log n}\le\phi(q)T(\log T)^2\sum_n\frac{|c_n|^2}{n^{2\al}}
\]
(since $c_1=0$). If $n$ has prime factorization $p_1^{e_1}\dotsb p_m^{e_m}$ with $e=e_1+\dotsb e_m$, then we can explicitly write
\[
c_n=\binom{e}{e_1,\dotsc,e_m}\prod_{i=1}^ma(p_i)^{e_i}
\]
so
\begin{align*}
\sum_n\frac{|c_n|^2}{n^{2\al}} &=\sum_n\binom{e}{e_1,\dotsc,e_m}^2\prod_{i=1}^m\frac{|a(p_i)|^{e_i}}{p_i^{2k_i\al}} \\
&\le k!\sum_n\binom{e}{e_1,\dotsc,e_m}\prod_{i=1}^m\frac{|a(p_i)|^{e_i}}{p_i^{2k_i\al}} = k!\left(\sum_{p\le N}|a(p)|^2p^{-2\al}\right)^k
\end{align*}
and we are done.
\end{proof}

\section{Point count bounds}
\label{pointcount}

In this section, we use Lemma~\ref{zerostoprimesum} and Lemma~\ref{sieve} to upper-bound the frequency with which an ordinate $t$ may have an abnormal number of zeros of $L(s,\chi)$ near it. We introduce some quick notation.

\begin{defn}
\label{range}
Let
\[
a(T,q)=\sqrt{\log q}(\log\log(qT))^2,\hspace{0.2in} b(T,q)=\frac{\log(qT)}{\log\log(qT)}.
\]
\end{defn}

We will also need the following elementary inequality.

\begin{lem}
\label{elem1}
Let $a(T,q)\le V\le b(T,q)$, $\eta=1/\log V$, $k=\left\lfloor\frac{V}{1+\eta}\right\rfloor$. Then we have
\[
k(\log(k\log\log qT)-2\log(\eta V)\le -V\log\left(\frac{V}{\log\log qT}\right)+2V\log\log V+V.
\]
\end{lem}

\begin{proof}
This calculation is Proposition 14 in~\cite{baladero}. We reproduce the proof here. Since $k\le V$, we have\[\log(k\log\log(qT))-2\log(\eta V)\le-\log\left(\frac{V}{\log\log(qT)}\right)+2\log\log V\le0\]but also $k\ge\frac{V}{1+\eta}-1\ge V(1-\eta)$, so\[k(\log(k\log\log(qT))-2\log(\eta V))\le V(1-\eta)\left(-\log\left(\frac{V}{\log\log(qT)}\right)+2\log\log V\right)\]\[\le-V\log\left(\frac{V}{\log\log(qT)}\right)+2V\log\log V+\eta V\log V=-V\log\left(\frac{V}{\log\log(qT)}\right)+2V\log\log V+V.\]
\end{proof}

All our statements about the distribution of zeros of $L(s,\chi)$ assume that $\chi$ is primitive, but of course Perron's formula for $M(x;q,a)$ includes the imprimitive characters $\pmod{q}$ as well. To deal with this, we just note that there is not much difference between the behavior of $\chi$ and $\chi_1$ where $\chi_1$ is the primitive character inducing $\chi$.

\begin{lem}
\label{nonprimerr}
For any character $\chi\pmod{q}$, let $\chi_1$ be the primitive character inducing $\chi$ and $q_1$ be the conductor of $\chi$. 
\begin{enumerate}
\item For any twisted Dirichlet polynomial $A(s,\chi)=\sum_{p\le N}a(p)\chi(p)p^{-s}$ and $\sg\ge1/2$, we have
\[
|A(\sg+it,\chi_1)-A(\sg+it,\chi)|\ll(\max_pa(p))\frac{\sqrt{\log q}}{\log\log q}.
\]
\item If $l_d(s,\chi)=\prod_{p|d}\left(1-\frac{\chi(p)}{p^s}\right)$, then we have
\[
\log|l_d(s,\chi)|=O\left(\frac{\sqrt{\log d}}{\log\log d}\right)
\]
and, as a consequence,
\[
\log|L(s,\chi)| - \log|L(s,\chi_1)| = O\left(\frac{\sqrt{\log q}}{\log\log q}\right).
\]
\end{enumerate}
\end{lem}

\begin{proof}
\begin{enumerate}
\item We can write
\[
|A(\sg+it,\chi_1)-A(\sg+it,\chi)|=\left|\sum_{p|q}\frac{a(p)\chi_1(p)}{p^{\sg+it}}\right|\le\max_pa(p)\sum_{p|q}\frac1{p^{1/2}}.
\]
We then have
\[
\sum_{p|q}\frac1{p^{1/2}}\le\sum_{p\le2\log q}\frac1{p^{1/2}}\ll\frac{\sqrt{\log q}}{\log\log q}.
\]

\item We can write
\[
\log|l_d(s,\chi)|=\sum_{p|d}\log\left|1-\frac{\chi(p)}{p^s}\right|\le \sum_{p|d}\frac1{p^{1/2}}=O\left(\frac{\sqrt{\log d}}{\log\log d}\right)
\]
as in part 1, and 
\[
L(s,\chi)=L(s,\chi_1)\prod_{p|q}\left(1-\frac{\chi_1(p)}{p^s}\right)=L(s,\chi_1)l_q(s,\chi_1).
\]
\end{enumerate}
\end{proof}

We are now ready to formulate the analog for Dirichlet $L$-functions of the key estimate in Soundararajan's proof. 

\begin{defn}
\label{vtypical}
Let $q\in\nats$ and let $\chi$ be a character $\pmod{q}$ induced by primitive $\chi_1\pmod{q_1}$. Let $T>e$, $0<\del\le1$, and $a(T,q)\le V\le b(T,q)$.

An ordinate $t\in[T,2T]$ is $(V,\del,\chi)$-typical of order $T$ if it satisfies the following conditions:

i. Let $y=(qT)^{1/V}$. For all $\sg\ge\frac12$, we have 
\[
\left|\sum_{n\le y}\frac{\chi_1(n)\Lambda(n)}{n^{\sg+it}\log n}\frac{\log(y/n)}{\log y}\right|\le2V.
\]

ii. Every sub-interval of $(t-1,t+1)$ of length $\frac{2\del\pi V}{\log(qT)}$ contains at most $(1+\del)V$ ordinates of zeroes of $L(s,\chi)$.

iii. Every sub-interval of $(t-1,t+1)$ of length $\frac{2\pi V}{\log V\log(qT)}$ contains at most $V$ ordinates of zeroes of $L(s,\chi)$.
\end{defn}

Note that this definition differs from the one used in~\cite{halupsu}, in that it sets $y=(qT)^{1/V}$ rather than $T^{1/V}$. When there is no risk of confusion, we will shorten ``$(V,\del,\chi)$-typical'' to ``$V$-typical''.

Our next two propositions give an upper bound for the size of any set of well-separated $V$-atypical ordinates. First we bound the size of any set of well-separated ordinates with a given accumulation of zeros in intervals of given size centered on the ordinates.

\begin{prop}
\label{zerocount}
Let $\chi$ be a character $\pmod{q}$ with conductor $q_1$, $T$ be large, $h$ be such that $0<h\le\sqrt T$, $a(T,q)\le V\le b(T,q)$, and $\{t_r^{\chi}\}$ be well-separated ordinates with $T<t_1^{\chi}<t_2^{\chi}<\dotsb<t_{R_{\chi}}^{\chi}\le 2T$ for all $\chi$, such that for all $\chi$ and $r$ we have
\[
N(t_r^{\chi}+h)-N(t_r^{\chi}-h,\chi)-\frac{h}{\pi}\log\frac{q_1t_r^{\chi}}{2\pi}\ge V+O(1).
\]
Let $R=\sum_{\chi}R_{\chi}$. Then
\[
R\ll\phi(q) T\exp\left(-V\log\frac{V}{\log\log(qT)}+2V\log\log V+O(V)\right).
\]
\end{prop}

\begin{proof}
By Lemma~\ref{zerostoprimesum}, we have for all $\Del\ge2$ that 
\begin{align*}
V+O(1) &\le N(t_r^{\chi}+h)-N(t_r^{\chi}-h,\chi)-\frac{h}{\pi}\log\frac{q_1t_r^{\chi}}{2\pi} \\
&\le\frac{\log(2qT)}{2\pi\Delta}+\left|\frac1{\pi}\sum_{p\le e^{2\pi\Del}}\frac{\chi_1(p)\log p}{p^{1/2+it_r^{\chi}}}\hat{F}_+\left(\frac{\log p}{2\pi}\right)\right|+O(\log\Del).
\end{align*}
Let $a(p)=\frac{\log p}{\pi}\hat{F}_+\left(\frac{\log p}{2\pi}\right)$, so that $|a(p)|\le4$ and
\[
\left|\sum_{p\le e^{2\pi\Del}}\frac{a(p)\chi_1(p)}{p^{1/2+it_r^{\chi}}}\right|\ge V-\frac{\log(2qT)}{2\pi\Del}+O(\log\Del)
\]
which implies by Lemma~\ref{nonprimerr} that
\[
\left|\sum_{p\le e^{2\pi\Del}}\frac{a(p)\chi(p)}{p^{1/2+it_r^{\chi}}}\right|\ge V-\frac{\log(2qT)}{2\pi\Del}+O\left(\log\Del+\frac{\sqrt{\log q}}{\log\log q}\right).
\]
Let $\eta=1/\log V$ and $\Del=\frac{(1+\eta)\log(qT)}{2\pi V}$. We have $\exp(2\pi\Del)=(qT)^{(1+\eta)/V}$ and $\log\Del\ll\log\log(qT)\le\sqrt{V}$. Also 
\[
\eta V=\frac{V}{\log V}\ge\frac{2\sqrt{\log q}(\log\log(qT))^2}{\log\log q+4\log\log\log(qT)}\ge\sqrt{\log q}
\]
so for $q$ sufficiently large, we have 
\[
\frac{\sqrt{\log q}}{\log\log q}\le\frac{\sqrt{\log q}}{100}\le\frac1{100}\eta V.
\]
On the other hand, for small $q$, we also have 
\[
\frac{\sqrt{\log q}}{\log\log q}\le\frac1{100}\eta V
\]
for $T$ sufficiently large. Therefore, as long as $qT$ is sufficently large, we have
\[
V-\frac{\log(2qT)}{2\pi\Del}+O\left(\log\Del+\frac{\sqrt{\log q}}{\log\log q}\right)\ge V-\frac{V\log(2qT)}{(1+\eta)\log(qT)}+\frac{\eta V}{100}+O(\sqrt{V})
\]
\[
\ge\frac{\eta V}{1+\eta}-\frac{\log 2}{(1+\eta)\log\log T}+\frac{\eta V}{100}+O(\sqrt{V})\ge\frac12\eta V.
\]
Let $k=\left\lfloor\frac{V}{1+\eta}\right\rfloor$, so that in fact $\exp(2\pi\Del)^k\le T$ and we can apply Proposition~\ref{sieve}. This gives
\[
R\left(\frac12\eta V\right)^{2k}\le\sum_{\chi}\sum_r\left|\sum_{p\le(qT)^{(1+\eta)/V}}\frac{a(p)}{p^{1/2+it_r^{\chi}}}\right|^{2k}\ll\phi(q) T(\log T)^2(Ck\log\log(qT))^k.
\]
Using Lemma~\ref{elem1}, this implies
\[
R\ll\phi(q)T\exp\left(-V\log\frac{V}{\log\log(qT)}+2V\log\log V+O(V)\right).
\]
\end{proof}

We now translate this bound into one for $V$-atypical ordinates.

\begin{prop}
\label{vtypicalbd}
Let $\chi$ be a character $\pmod{q}$, let $T$ be large, let $a(T,q)\le V\le b(T,q)$, and let $\{t_r^{\chi}\}$ be well-separated $V$-atypical ordinates with $T<t_1^{\chi}<t_2^{\chi}<\dotsb<t_{R_{\chi}}^{\chi}\le 2T$ for all $\chi$. Let $R=\sum_{\chi}R_{\chi}$. Then
\[
R\ll T\exp\phi(q)\left(-V\log\left(\frac{V}{\log\log(qT)}\right)+2V\log\log V+O(V)\right).
\]
\end{prop}

\begin{proof}
Let $R_1$ be the number of $t_r^{\chi}$ in the list failing criterion (i). For such $t_r^{\chi}$ there is a $\sg_r^{\chi}\ge\frac12$ such that
\[
\left|\sum_{n\le y}\frac{\chi_1(n)\Lambda(n)}{n^{\sg_r^{\chi}+it}\log n}\frac{\log(y/n)}{\log y}\right|>2V.
\]
The sum over $n=p^{\al}$ with $\al\ge2$ is 
\[
\left|\sum_{p^{\al}\le y,\al\ge2}\frac{\chi_1(p^{\al})}{p^{\sg_r^{\chi}+it}}\frac{\log(y/n)}{\log y}\right|
\le\sum_{p\le\sqrt y}\frac1{p}+\sum_{p^{\al}\le y,\al\ge3}\frac1{p^{\al/2}}
\]
\[
\ll\log\log y\ll\log\log(qT)\le\sqrt{V}.
\]
The error $\frac{\sqrt{\log q}}{\log\log q}$ from subtracting off the contribution from imprimitive characters is a small fraction of $V$ for $qT$ sufficiently large, as argued in the proof of Proposition~\ref{zerocount}, so it suffices to count $t_r^{\chi}$ such that
\[
\left|\sum_{p\le y}\frac{\chi(p)}{p^{\sg_r+it}}\frac{\log(y/p)}{\log y}\right|\ge V.
\]
Since $y^k\le qT$ for all $k\le V$, applying Proposition~\ref{sieve} gives
\[
R_1V^{2k}\le\sum_{\chi}\sum_{r\le R}\left|\sum_{p\le y}\frac{\chi(p)}{p^{\sg_r+it}}\frac{\log(y/p)}{\log y}\right|^{2k}
\ll \phi(q)T(\log T)^2k!\left(\sum_{p\le y}\frac{\log^2(y/p)}{p\log^2y}\right)^k
\]
and we have
\[
\sum_{p\le y}\frac{\log^2(y/p)}{p\log^2y}\le\sum_{p\le y}\frac1p\ll\log\log y\ll\log\log(qT),
\]
so
\[
R_1\ll \phi(q)T\exp\left(-V\log\frac{V}{\log\log(qT)}+O(V)\right).
\]
Next let $R_2$ be the number of $t_r^{\chi}$ in the list failing criterion (ii). For each such $t_r$ there is some $t_r'$ with $|t_r-t_r'|\le1$ and
\[
N\left(t_r'+\frac{\pi\del V}{\log(qT)},\chi\right)-N\left(t_r'-\frac{\pi\del V}{\log(qT)},\chi\right)>(1+\del)V
\]
from which
\[
N\left(t_r'+\frac{\pi\del V}{\log(qT)},\chi\right)-N\left(t_r'-\frac{\pi\del V}{\log(qT)},\chi\right)-\frac{\del V}{\log(qT)}\log\left(\frac{q_1t_r'}{2\pi}\right)\ge V+O(1).
\]
Applying Proposition~\ref{zerocount} on the sets $\{t_{3s+j}^{\chi}\}$ for $j=0,1,2$ gives the desired bound on $R_2$. The computation for $R_3$, the number of $t_r^{\chi}$ failing criterion (iii), is effectively the same.
\end{proof}

We will additionally find it convenient to obtain the following absolute bound on the amount by which $N(t+h,\chi)-N(t-h,\chi)$ can deviate from average. This is an analog of Theorem 1 of~\cite{goldgon}, and is also proven as Propositon 7 of~\cite{halupsu}.

\begin{prop}
\label{goldgonanalog}
Let $qt$ be sufficiently large, $0<h\le\sqrt{t}$, and $\chi$ be a primitive character $\pmod{q}$. Then
\[
\left|N(t+h,\chi)-N(t-h,\chi)-\frac{h}{\pi}\log\frac{qt}{2\pi}\right|\le\frac{\log(qt)}{2\log\log(qt)}+\left(\frac12+o(1)\right)\frac{\log(qt)\log\log\log(qt)}{(\log\log(qt))^2}.
\]
\end{prop}

\begin{proof}
We apply Proposition~\ref{zerostoprimesum} with $\del=\frac1{\pi}\log\frac{\log(qt)}{\log\log(qt)}$. Since
\[
\left|\frac1{\pi}\Re\left(\sum_{p\le e^{2\pi\Del}}\frac{\chi(p)\log p}{p^{1/2+it}}\hat{F}_+\left(\frac{\log p}{2\pi}\right)\right)\right|
\ll\sum_{p\le e^{2\pi\Del}}\frac1{\sqrt p}\ll\frac{e^{\pi\Del}}{\Del},
\]
we can compute
\[
\left|N(t+h,\chi)-N(t-h,\chi)-\frac{h}{\pi}\log\frac{qt}{2\pi}\right|
\]
\[
\le\frac{\log(qt)}{2(\log\log(qt)-\log\log\log(qt))}+O\left(\frac{\log(qt)/\log\log(qt)}{\log\log(qt)-\log\log\log(qt)}\right)
\]
\[
=\frac{\log(qt)}{2\log\log(qt)}\sum_{k=0}^{\infty}\left(\frac{\log\log\log(qt)}{\log\log(qt)}\right)^k+O\left(\frac{\log(qt)}{(\log\log(qt))^2}\right)
\]
\[
=\frac{\log(qt)}{2\log\log(qt)}+\frac{\log(qt)\log\log\log(qt)}{2(\log\log(qt))^2}(1+o(1)).
\]
\end{proof}

The above bound allows us to state, for a given $t$, a choice of $V$ for which $t$ is guaranteed to be $V$-typical.

\begin{prop}
\label{extremev}
For $V$ between $(1/2+o(1))\log(qT)/\log\log(qT)$ and $\log(qT)/\log\log(qT)$, all ordinates $t\in[T,2T]$ are $V$-typical of order $T$.
\end{prop}

\begin{proof}
We have
\[
f(u)=\sum_{n\le u}\frac{\Lam(n)\chi_1(n)}{\sqrt{n}\log n}\ll\frac{\sqrt{u}}{\log u}
\]
and consequently
\[
\sum_{n\le y}\frac{\Lam(n)\chi_1(n)}{\sqrt{n}\log n}\log\left(\frac yn\right)=\sum_{n\le y}\frac{\Lam(n)\chi_1(n)}{\sqrt{n}\log n}\sum_{n\le j\le y}\log\left(\frac{j+1}j\right)
\]
\[
=\sum_{j\le y}\sum_{n\le j}\frac{\Lam(n)\chi_1(n)}{\sqrt{n}\log n}\log\left(\frac{j+1}j\right)=\sum_{j\le y}\int_j^{j+1}\sum_{n\le j}\frac{\Lam(n)\chi_1(n)}{\sqrt{n}\log n}\frac{du}{u}
\]
\[
=\int_1^yf(u)\frac{du}u\ll\frac{\sqrt{y}}{\log y}.
\]
We conclude that
\[
\left|\sum_{n\le y}\frac{\chi_1(n)\Lambda(n)}{n^{\sg_r+it}\log n}\frac{\log(y/n)}{\log n}\right|\ll\frac{\sqrt{y}}{(\log y)^2}=\frac{V^2(qT)^{1/V}}{(\log(qT))^2}\le\frac{\log(qT)}{(\log\log(qT))^2}=o(V).
\]
The other two follow directly from Proposition~\ref{goldgonanalog}.
\end{proof}

\section{$L(s,\chi)$ bounds}
\label{Lfuncbds}

Our next goal is to find lower bounds for $L(s,\chi)$ given that $s$ has $V$-typical imaginary part, which we will obtain from bounds for the logarithmic derivative $\frac{L'}{L}(s,\chi)$. Let $F(s,\chi)=\sum_{\rho}\Re\left(\frac{1}{s-\rho}\right)$ where $\rho$ runs over the nontrivial zeroes of $L(s,\chi)$. We can write $L'/L$ in terms of $F$.

\begin{prop}
\label{logderinF}
Let $\chi$ be a primitive character $\pmod{q}$, $T$ be sufficiently large, $\frac12\le\sg\le2$, $T\le t\le 2T$, and $\sg$ be such that $L(\sg+it,\chi)\neq0$. Then
\[
\Re\frac{L'}{L}(\sg+it,\chi)=F(\sg+it,\chi)-\frac12\log(qT)+O(1).
\]
\end{prop}

\begin{proof}
We start with Equation 12.17 of~\cite{davenport}, which states that for primitive $\chi$,
\[
\frac{L'}{L}(s,\chi)=-\frac12\log\frac q{\pi}-\frac12\frac{\Gam'}{\Gam}\left(\frac{s+\kap}2\right)+B(\chi)+\sum_{\rho}\left(\frac1{s-\rho}+\frac1{\rho}\right)
\]
where $\Re(B(\chi))=-\sum_{\rho}\Re(1/\rho)$. From Stirling's formula for $\Gam'/\Gam$, we have
\begin{align*}
\Re\frac{L'}{L}(s,\chi) &=-\frac12\log\frac q{\pi}-\frac12\Re\frac{\Gam'}{\Gam}\left(\frac{\sg+it+\kap}2\right)+\Re(B(\chi))+\sum_{\rho}\Re\left(\frac1{\sg+it-\rho}+\frac1{\rho}\right) \\
&=-\frac12\log\frac q{\pi}-\frac12\log|\sg+it+\kap|+O(|\sg+it+\kap|^{-1})+F(\sg+it,\chi) \\
&=F(\sg+it,\chi)-\frac12\log(qT)+O(1).
\end{align*}
\end{proof}

We also need $L'/L$ in terms of a sum over primes together with another sum over zeros of $L$.

\begin{prop}
\label{eqnwithz}
Let $\chi$ be a primitive character $\pmod{q}$ and $y\ge1$. Let $z\in\cplx$ be such that $\Re(z)\ge0$ and $T\le\Im(z)\le2T$, and assume that $z$ is not a pole of $\frac{L'}{L}(s,\chi)$. Then
\[
\sum_{n\le y}\frac{\chi(n)\Lam(n)}{n^z}\log\left(\frac yn\right)=-\frac{L'}{L}(z,\chi)\log y-\left(\frac{L'}{L}\right)'(s,\chi)-\sum_{\rho}\frac{y^{\rho-z}}{(\rho-z)^2}+O(T^{-1}).
\]
\end{prop}

\begin{rem}
This is Proposition 13 of~\cite{halupsu}, with the condition $y\le T$ removed.
\end{rem}

\begin{proof}
Proposition 12 in~\cite{halupsu} states that for $\Re(z)>1/2$,
\[
\sum_{n\le y}\frac{\chi(n)\Lam(n)}{n^z}\log\left(\frac yn\right)=-\frac{L'}{L}(z,\chi)\log y-\left(\frac{L'}{L}\right)'(z,\chi)-\sum_{\rho}\frac{y^{\rho-z}}{(\rho-z)^2}-\sum_{n\ge0}\frac{y^{-2n-\kap-z}}{(z+2n+\kap)^2}.
\]
The proof is as follows. We begin with the identity
\[
\frac1{2\pi i}\int_{c-i\infty}^{c+i\infty}-\frac{L'}{L}(z+w,\chi)\frac{y^w}{w^2}dw=\sum_{n\le y}\frac{\Lam(n)\chi(n)}{n^z}\log\frac yn,
\]
which is valid for $\Re(z)>1/2$ and $c>1/2$, and can be obtained by expanding $-\frac{L'}{L}(z+w,\chi)$ as a Dirichlet series and integrating term by term. Then we shift the line of integration to $\Re(w)=-2N-1-\kap-\Re(z)$ where $N$ is a positive integer, picking up the residue $-\frac{L'}{L}(z,\chi)\log y-\left(\frac{L'}{L}\right)'(z,\chi)$ from the pole at $0$ and the sums $-\sum_{\rho}\frac{y^{\rho-z}}{(\rho-z)^2}$ and $-\sum_{0\le n\le2N+\kap}\frac{y^{-2n-\kap-z}}{(z+2n+\kap)^2}$ from the nontrivial and trivial zeroes of $L$ respectively. Since\\ $\frac{L'}{L}(s,\chi)\ll\log(q|s|)$ for $\Re(s)\le-1$ excluding circles of radius $1/2$ around the trivial zeroes (see page 116 of~\cite{davenport}), the integral $\int_{c'-i\infty}^{c'+i\infty}-\frac{L'}{L}(z+w,\chi)\frac{y^w}{w^2}dw$ approaches $0$ as $N$ approaches infinity, giving the identity above.

Finally we note that since $\Re(-2n-\kap-z)\le0$ for $n\ge0,\kap\in\{0,1\}$, and $\Re(z)\ge0$, it is clear that
\[
\left|\sum_{n\ge0}\frac{y^{-2n-\kap-z}}{(z+2n+\kap)^2}\right|\le\sum_{n\ge0}\frac1{(T+2n+\kap)^2}\ll\frac1T.
\]
\end{proof}

From the two previous propositions, we can write an explicit expression for a lower bound for $\log|L(\sg+it,\chi)|$.

\begin{prop}
\label{initineq}
Let $\chi$ be a  character $\pmod{q}$ induced by $\chi_1\pmod{q_1}$. Let $T$ be sufficiently large and $t\in[T,2T]$. Then we have for all $\frac12\le\sg\le2$ and $y\ge2$ that
\[
\log|L(\sg+it,\chi)|\ge\Re\left(\sum_{n\le y}\frac{\Lam(n)\chi_1(n)}{n^{\sg+it}\log n}\frac{\log(y/n)}{\log y}\right)-\left(1+\frac{y^{1/2-\sg}}{(\sg-1/2)\log y}\right)\frac{F(\sg+it,\chi)}{\log y}
\]
\[
+O\left(\frac{\sqrt{\log q}}{\log\log q}\right).
\]
\end{prop}

\begin{proof}
First assume $\chi$ is primitive. We integrate the equation from Proposition~\ref{eqnwithz} from $z=\sg+it$ to $z=2+it$. This gives
\[
\sum_{n\le y}\Lam(n)\chi(n)\log\frac yn\left(\frac{n^{-2-it}}{-\log n}-\frac{n^{-\sg-it}}{-\log n}\right)
\]
\[
=(-\log L(2+it,\chi)+\log L(\sg+it,\chi))\log y-\frac{L'}{L}(2+it,\chi)+\frac{L'}{L}(\sg+it,\chi)
\]
\[
-\sum_{\rho}\int_{\sg}^2\frac{y^{\rho-u-it}}{(\rho-u-it)^2}du+O(T^{-1})
\]
or, rearranging and absorbing small terms into the error term,
\[
(\log y)\log L(\sg+it,\chi)=\sum_{n\le y}\frac{\Lam(n)\chi(n)}{n^{\sg+it\log n}}\log\frac yn-\frac{L'}{L}(\sg+it,\chi)+\sum_{\rho}\int_{\sg}^2\frac{y^{\rho-u-it}}{(\rho-u-it)^2}du+O(\log y).
\]
We divide by $\log y$, take real parts, and use Proposition~\ref{logderinF} to turn this into
\[
\log|L(\sg+it,\chi)|=\Re\sum_{n\le y}\frac{\Lam(n)\chi(n)}{n^{\sg+it}\log n}\frac{\log(y/n)}{\log y}-\frac{F(\sg+it)}{\log y}+\frac{\log(qT)}{2\log y}
\]
\[
+\frac1{\log y}\Re\sum_{\rho}\int_{\sg}^2\frac{y^{\rho-u-it}}{(\rho-u-it)^2}du+O(1)
\]
where
\[
\left|\sum_{\rho}\int_{\sg}^2\frac{y^{\rho-u-it}}{(\rho-u-it)^2}du\right| \le\sum_{\rho}\frac1{|\rho-\sg-it|^2}\int_{\sg}^2y^{1/2-u}du
\]
\[
\le\frac{y^{1/2-\sg}}{\log y}\sum_{\rho}\frac1{|\rho-\sg-it|^2}=\frac{y^{1/2-\sg}}{(\sg-1/2)\log y}F(\sg+it,\chi).
\]
This gives the inequality for $\chi$ primitive. Otherwise we have
\[
\log|L(s,\chi)|
=\log|L(s,\chi_1)|+O\left(\frac{\sqrt{\log q}}{\log\log q}\right)
\]
by Lemma~\ref{nonprimerr}.
\end{proof}

Now we assume that $t$ is $V$-typical and get two explicit bounds for $\log|L(\sg+it,\chi)|$, depending on the size of $\sg$. The first is for $\sg$ not too close to $1/2$, and is like Proposition 15 of~\cite{halupsu}, with the condition $T\ge q$ removed.

\begin{prop}
\label{largesglogbound}
Let $\chi$ be a  character $\pmod{q}$. Let $T$ be sufficiently large, $a(T,q)\le V\le b(T,q)$, and $t\in[T,2T]$ be $(V,\del,\chi)$-typical of order $T$. Then for $\sg$ such that $\frac12+\frac{V}{\log(qT)}\le\sg\le2$ and some $C>0$, we have
\[
\log|L(\sg+it,\chi)|\ge-C\left(\frac{V}{\del}+\sqrt{\frac{\log q}{\log\log q}}\right).
\]
\end{prop}

\begin{proof}
For $y=(qT)^{1/V}$, we have
\[
\frac{y^{1/2-\sg}}{(\sg-1/2)\log y}\le\frac{\exp\left(-\frac{V}{\log(qT)}\cdot\frac{\log qT}{V}\right)}{\frac{V}{\log(qT)}\frac{\log qT}{V}}
=\frac{\exp(-1)}{1}<1
\]
so by Proposition~\ref{initineq}, given that $t$ is $(V,\del,\chi)$-typical, we have
\[
\log|L(\sg+it,\chi)|\ge-2V-2\frac{V}{\log(qT)}F(\sg+it,\chi)+O\left(\sqrt{\frac{\log q}{\log\log q}}\right).
\]
The proof of Proposition 15 in~\cite{halupsu} states that $F(\sg+it,\chi)=O\left(\frac{\log qT}{\del}\right)$; we reproduce this computation here. For 
$0\le n\le N=\left\lfloor\frac{\log(qT)}{4\pi\del V}\right\rfloor$, let $I_n$ be the set of zeros $\rho=1/2+i\gam$ such that
\[
\frac{2\pi n\del V}{\log(qT)}\le|t-\gam|\le\frac{2\pi(n+1)\del V}{\log(qT)}.
\]
Note that $\sum_{n=0}^N\frac{a}{a^2+c^2n^2}\le\frac1a+\int_0^{\infty}\frac{a}{a^2+c^2t^2}dt=\frac1a+\frac{\pi}{2c}$. Then since $t$ is $V$-typical, we use Property 2 of Definition~\ref{vtypical} to get
\[
\sum_{\gam\in I_n}\Re\left(\frac1{\sg+it-1/2-i\gam}\right) = \sum_{\gam\in I_n}\frac{\sg-1/2}{(\sg-1/2)^2+(t-\gam)^2}
\]
\[
\le2(1+\del)V\sum_{n=0}^N\frac{\sg-1/2}{(\sg-1/2)^2+\left(\frac{2\pi n\del V}{\log(qT)}\right)^2} \le 2(1+\del)V\left(\frac1{\sg-1/2}+\frac{\log(qT)}{4\del V}\right)
\]
\[
\le2(1+\del)V\left(\frac{\log(qT)}V+\frac{\log(qT)}{4\del V}\right)=O(\log(qT)/\del).
\]
For the remaining zeros, we note that if $|t-\gam|\ge1/2$ and $1/2\le\sg\le2$ then
\[
\frac{(\sg-1/2)^2+(t-\gam)^2}{1+(t-\gam)^2}\ge\frac15\ge\frac{\sg-1/2}{8}
\]
so citing Equation 16.3 of~\cite{davenport}, we have
\[
\sum_{|t-\gam|\ge1/2}\Re\left(\frac1{\sg+it-1/2-i\gam}\right) = \sum_{|t-\gam|\ge1/2}\frac{\sg-1/2}{(\sg-1/2)^2+(t-\gam)^2}
\]
\[
\le\sum_{|t-\gam|\ge1/2}\frac{8}{1+(t-\gam)^2}\le\sum_{\rho}\frac{8}{1+(t-\Im(\rho))^2}\ll\log(qt).
\]
Combining these gives $F(\sg+it,\chi)=O\left(\frac{\log qT}{\del}\right)$ as desired, so the middle term $-2\frac{V}{\log(qT)}F(\sg+it,\chi)$ is $O(V/\del)$, and we are done.
\end{proof}

If $\sg$ is very close to $1/2$, we instead compute the following.

\begin{prop}
\label{smallsglogbound}
Let $\chi$ be a character $\pmod{q}$, $T$ be sufficiently large, $a(T,q)\le V\le b(T,q)$, and $t\in[T,2T]$ be $V$-typical of order $T$. Then we have for all $\frac12<\sg\le\sg_0=\frac12+\frac{V}{\log(qT)}$ that
\[
\log|L(\sg+it,\chi)|\ge\log|L(\sg_0+it,\chi)|-V\log\frac{\sg_0-1/2}{\sg-1/2}-2(1+\del)V\log\log V+O\left(\frac{V}{\del^2}+\frac{\sqrt{\log q}}{\log\log q}\right).
\]
\end{prop}

\begin{proof}
This is Proposition 16 of~\cite{halupsu}; we reproduce the proof here. First assume $\chi$ is primitive. Using Proposition~\ref{logderinF}, we have
\[
\log|L(\sg_0+it,\chi)|-\log|L(\sg+it,\chi)|=\int_{\sg}^{\sg_0}\Re\frac{L'}{L}(u+it,\chi)du\le\int_{\sg}^{\sg_0}F(u+it,\chi)du
\]
\[
=\sum_{\gam}\int_{\sg}^{\sg_0}\frac{u-1/2}{(u-1/2)^2+(t-\gam)^2}du=\frac12\sum_{\gam}\log\frac{(\sg_0-1/2)^2+(t-\gam)^2}{(\sg-1/2)^2+(t-\gam)^2}.
\]
Again we separate the sum according to the distance of $\gam$ from $t$. By Property 3 of Definition~\ref{vtypical}, we have
\begin{align*}
\frac12\sum_{|t-\gam|<\pi V/(\log V\log(qT))}\log\frac{(\sg_0-1/2)^2+(t-\gam)^2}{(\sg-1/2)^2+(t-\gam)^2} &\le\frac12\sum_{|t-\gam|<\pi V/(\log V\log(q_1T))}\log\frac{(\sg_0-1/2)^2}{(\sg-1/2)^2} \\
&\le V\log\frac{\sg_0-1/2}{\sg-1/2}.
\end{align*}
Next, for $0\le n\le N=\left\lfloor\frac{\log(qT)}{4\pi\del V}\right\rfloor$, let $J_n$ be the set of $\gam$ such that
\[
\left(2\pi\del n+\frac{\pi}{\log V}\right)\frac{V}{\log(qT)}\le|t-\gam|\le\left(2\pi\del(n+1)+\frac{\pi}{\log V}\right)\frac{V}{\log(qT)}.
\]
Then, by Property 2 of Definition~\ref{vtypical},
\[
\frac12\sum_{\gam\in J_n}\log\frac{(\sg_0-1/2)^2+(t-\gam)^2}{(\sg-1/2)^2+(t-\gam)^2} \le 2(1+\del)V\cdot\frac12\sum_{n=0}^N\log\frac{1+(2\pi\del n+\pi/\log V)^2}{(2\pi\del n+\pi/\log V)^2}
\]
\[
=(1+\del)V\log\left(1+\frac{\log V^2}{\pi^2}\right)+(1+\del)V\sum_{n=1}^N\log\left(1+\frac1{(2\pi\del n+\pi/\log V)^2}\right)
\]
\[
\le2(1+\del)V\log\log V+O(V/\del^2).
\]
Finally, since $\frac1{(t-\gam)^2}\le\frac{5}{1+(t-\gam)^2}$ for $|t-\gam|\ge1/2$, we see, again by Equation 16.3 of~\cite{davenport}, that
\[
\frac12\sum_{|t-\gam|\ge1/2}\log\frac{(\sg_0-1/2)^2+(t-\gam)^2}{(\sg-1/2)^2+(t-\gam)^2} \le \frac12\sum_{|t-\gam|\ge1/2}\log\left(1+\frac{(\sg_0-1/2)^2}{(t-\gam)^2}\right)
\]
\[
\le \frac12\sum_{|t-\gam|\ge1/2}\frac{(\sg_0-1/2)^2}{(t-\gam)^2}\ll\frac12\left(\frac{V}{\log(qT)}\right)^2\log(qT)\le\frac{V}{2\log\log qT}.
\]
This gives the bound for $\chi$ primitive. For $\chi$ imprimitive, we get the same inequality by Lemma~\ref{nonprimerr}.
\end{proof}

Now we put Propositions~\ref{smallsglogbound} and~\ref{largesglogbound} together to get a bound of the right size for $\left|x^zL(z,\chi)^{-1}\right|$ for all $z$ in the range we need.

\begin{prop}
\label{completebd}
Let $\chi$ be a character $\pmod{q}$. Let $t$ be sufficiently large, $x\ge t$, $a(t/2,q)\le V\le b(t/2,q)$ so that $t$ is $V'$-typical of order $T'$, and $V\ge V'$.

Then for $z$ such that $V'\le(\Re(z)-1/2)\log x\le V$ and $|\Im(z)|=t$, we have
\[
\left|x^zL(z,\chi)^{-1}\right|\le\sqrt{x}\exp\left(V\log\frac{\log x}{\log(qt)}+2(1+\del)V\log\log V+O\left(\frac{V}{\del^2}+\frac{\log x}{\log\log x}\right)\right).
\]
\end{prop}

\begin{rem}
This is Proposition 19 of~\cite{halupsu}, with the condition $t\ge q$ removed.
\end{rem}

\begin{proof}
If $\Re(z)\le\frac12+\frac{V'}{\log(qT')}$, we can apply Proposition~\ref{smallsglogbound} to get
\begin{align*}
-\log|L(z,\chi)| &\le V'\log\frac{V'/\log(qT')}{\Re(z)-1/2}+2(1+\del)V'\log\log V'+O\left(\frac{V}{\del^2}+\frac{\sqrt{\log q}}{\log\log q}\right) \\
&\le V'\log\frac{V'/\log(qT')}{V'/\log x}+2(1+\del)V'\log\log V'+O\left(\frac{V}{\del^2}+\frac{\sqrt{\log q}}{\log\log q}\right) \\
&=  V'\log\frac{\log x}{\log(qT')}+2(1+\del)V'\log\log V'+O\left(\frac{V}{\del^2}+\frac{\sqrt{\log q}}{\log\log q}\right).
\end{align*}
If $\Re(z)>\frac12+\frac{V'}{\log(qT')}$, we can apply Proposition~\ref{largesglogbound}, from which it is clear that $-\log|L(z,\chi)|$ still satisfies the above bound. We conclude that
\begin{align*}
\log|x^zL(z,\chi)^{-1}| &= \Re(z)\log x-\log|L(z,\chi)| \\
&\le \frac12\log x+V+V'\log\frac{\log x}{\log(qT')}+2(1+\del)V'\log\log V'+O\left(\frac{V}{\del^2}+\frac{\sqrt{\log q}}{\log\log q}\right) \\
&\le \frac12\log x+V\log\frac{\log x}{\log(qt)}+2(1+\del)V\log\log V+O\left(\frac{V}{\del^2}+\frac{\sqrt{\log q}}{\log\log q}\right).
\end{align*}
\end{proof}

When $t$ is small, the $V$-typicality of $t$ becomes less useful, so we supplement with the following simple bound. It is similar to Proposition 18 of~\cite{halupsu}, but with the $q$-dependence appearing explicitly in the bound.

\begin{prop}
\label{trivial}
Let $x$ and $T$ be large and $\sg=\frac12+\frac1{\log x}$. Then there exists a $C>0$ such that for all $|t|\le T$ and  $\chi\pmod{q}$ we have
\[
|L(\sg+it,\chi)|\ge\exp(-C\log(qT)\log\log x).
\]
\end{prop}

\begin{proof}
Assume first that $\chi$ is primitive. From Equation 16.14 of~\cite{davenport}, we can write
\[
\int_{\sg+it}^{2+it}\frac{L'}{L}(s+it,\chi)ds=\int_{\sg+it}^{2+it}\left(\sum_{\substack{\rho \\ |\Im(s)-\Im(\rho)|\le1}}\frac1{s-\rho}+O(\log(q(|\Im(s)|+2)))\right)ds
\]
which implies
\[
\log|L(2+it,\chi)|-\log|L(\sg+it,\chi)|
\]
\[
=\sum_{\substack{\rho \\ |t-\Im(\rho)|\le1}}\log|2+it-\rho|-\sum_{\substack{\rho \\ |t-\Im(\rho)|\le1}}\log|\sg+it-\rho|+O(\log(q(|t|+2))).
\]
We can see that
\[
\sum_{\substack{\rho \\ |t-\Im(\rho)|\le1}}\log|2+it-\rho|\ll N(t+1,\chi)-N(t-1,\chi)\ll\log(q|t|)
\]
and since $|\sg+it-\rho|=\left|\frac1{\log x}+i(t-\Im(\rho))\right|\ge\frac1{\log x}$, we get
\[
\sum_{\substack{\rho \\ |t-\Im(\rho)|\le1}}\log|\sg+it-\rho|^{-1}\ll\log(q|t|)\log\log x.
\]
This gives the desired inequality for $-\log|L(\sg+it,\chi)|$. For $\chi$ imprimitive with conductor $q_1$, following the same approximation procedure as before, we can write
\[
\log|L(\sg+it,\chi)|^{-1}\ll\log(q_1|t|)+O\left(\frac{\sqrt{\log q}}{\log\log q}\right).
\]
\end{proof}

We will find it convenient to write out the results of Propositions~\ref{smallsglogbound} and~\ref{largesglogbound} for $V$ such that $t$ is guaranteed to be $V$-typical, as in Proposition 17 of~\cite{halupsu} with the condition $|t|\ge q$ removed.

\begin{prop}
\label{genericbd}
Let $\chi$ be a character $\pmod{q}$, $|t|$ sufficiently large, and $\frac12<\sg\le2$. Then
\[
\log|L(\sg+it,\chi)|\ge-\frac{\log(q|t|)}{\log\log(q|t|)}\log\frac1{\sg-1/2}-3\frac{\log(q|t|)\log\log\log(q|t|)}{\log\log(q|t|)}.
\]
\end{prop}

\begin{proof}
Let $V=\frac{\log(q|t|)}{\log\log(q|t|)}$ and $\del=1/2$, so that by Proposition~\ref{extremev} $t$ is $V$-typical of order $|t|$. Then by Proposition~\ref{smallsglogbound} and Proposition~\ref{largesglogbound}, we have
\begin{align*}
\log|L(\sg+it,\chi)| &\ge-V\log\frac{V/\log(q|t|)}{\sg-1/2}-2(1+\del)V\log\log V+O\left(\frac{V}{\del^2}+\frac{\sqrt{\log q}}{\log\log q}\right) \\
&\ge-\frac{\log(q|t|)}{\log\log(q|t|)}\log\frac1{\sg-1/2}-2\frac{\log(q|t|)\log\log\log(q|t|)}{\log\log(q|t|)}+O\left(\frac{\log(q|t|)}{\log\log(q|t|)}\right) \\
&\ge-\frac{\log(q|t|)}{\log\log(q|t|)}\log\frac1{\sg-1/2}-3\frac{\log(q|t|)\log\log\log(q|t|)}{\log\log(q|t|)}.
\end{align*}
\end{proof}

\section{Proof of theorem}
\label{finalthm}

We proceed to our estimation of $\int\frac{(x/d)^s}{sL(s,\chi)l_d(s,\chi)}ds$. For the sake of brevity, we write $x$ in place of $x/d$. We first introduce some notation.

\begin{defn}
Let $A(x,\chi)=\frac1{2\pi i}\int_{1+1/\log x-i\lfloor x\rfloor}^{1+1/\log x+i\lfloor x\rfloor}\frac{x^s}{sL(s,\chi)l_d(s,\chi)}ds$.
\end{defn}

\begin{defn}
We set $K=\left\lfloor\frac{\log x}{\log 2}\right\rfloor$, $l=\left\lfloor(\log x)^{1/2}(\log\log x)^c\right\rfloor$ (where $c$ will be determined later) if $q\le\exp(\sqrt{\log x})$ and otherwise $l=C$ for some large constant $C$, and $T_k=2^k$ for $l\le k\le K$.

For any $\chi$, $k$ with $l\le k<K$, and $n$ with $T_k\le n<2T_k$, let $V_n^{\chi}$ be the smallest integer in the interval $a(T_k,q)\le V\le b(T_k,q)$ such that all points in $[n,n+1]$ are $(V_n^{\chi},\del,\chi)$-typical ordinates of order $T_k$.
\end{defn}

We are going to split the line segment of integration into dyadic intervals, then, within each interval, deform the line of integration according to the $V$-typicality of the ordinates inside.

\begin{lem}
Let $x\ge2$ and $c>1$. Let $\chi$ be a  character $\pmod{q}$ and $0<\del\le1$. Then
\[
\frac{|A(x,\chi)|}{\sqrt x}\ll_{\del}\exp\left((\log x)^{1/2}(\log\log x)^{c+1+\del}\right)+B(x,\chi)
\]
where
\[
B(x,\chi)=\sum_{n=T_{l}}^{T_K-1}\frac1n\exp\left(V_n^{\chi}\log\left(\frac{\log x}{\log(qn)}\right)+2(1+2\del)V_n^{\chi}\log\log V_n^{\chi}+D\sqrt{\frac{\log x}{\log\log x}}\right).
\]
\end{lem}

\begin{proof}
We deform the path of integration for $L(s,\chi)$ as follows. It is symmetric across the real axis. In the upper half-plane, it consists of the following five groups of segments:
\begin{enumerate}
\item A vertical segment from $\frac12+\frac1{\log x}$ to $\frac12+\frac1{\log x}+iT_{l}$ if $q\le\exp(\sqrt{\log x})$, and otherwise from $\frac12+\eps_1$ to $\frac12+\eps_1+iT_l$ where $\eps_1$ is some positive constant,

\item Vertical segments from $\frac12+\frac{V_n^{\chi}}{\log x}+in$ to $\frac12+\frac{V_n^{\chi}}{\log x}+i(n+1)$,

\item A horizontal segment from $\frac12+\frac1{\log x}+iT_{l}$ to $\frac12+\frac{V_{l}^{\chi}}{\log x}+iT_{l}$ if $q\le\exp(\sqrt{\log x})$, and otherwise from $\frac12+\eps_1+iT_l$ to $\frac12+\frac{V_{l}^{\chi}}{\log x}+iT_{l}$

\item Horizontal segments from $\frac12+\frac{V_n^{\chi}}{\log x}+i(n+1)$ to $\frac12+\frac{V_{n+1}^{\chi}}{\log x}+i(n+1)$,

\item A horizontal segment from $\frac12+\frac{V_{T_K-1}^{\chi}}{\log x}+iT_K$ to $\frac12+\frac{1}{\log x}+iT_K$.
\end{enumerate}
We now estimate the contribution from each segment. By Lemma~\ref{nonprimerr}, the factor $l_d(s,\chi)$ contributes a factor of at most $\exp(c\sqrt{\log d}(\log\log d)^{-1})$ to all integrands.

By Proposition~\ref{trivial}, in the case $q\le\exp(\sqrt{\log x})$, segment 1 gives a contribution of 
\[
\frac1{2\pi}\left|\int_{1/2+1/\log x}^{1/2+1/\log x+iT_l}\frac{x^s}{sL(s,\chi)l_d(s,\chi)}ds\right|
\]
\[
\le\frac e{2\pi}\sqrt{x}\exp(c\sqrt{\log d}(\log\log d)^{-1})\max_{|t|\le T_l}|L(1/2+1/\log x+it,\chi)|^{-1}\int_0^{T_l}\frac{dt}{\sqrt{1/4+t^2}}
\]
\[
\le 2\sqrt{x}\exp(c\sqrt{\log d}(\log\log d)^{-1})\log T_{l}\max_{|t|\le T_{l}}\left|L\left(\frac12+\frac1{\log x}+it,\chi\right)\right|^{-1}
\]
\[
\ll2\sqrt{x}\log T_{l}\exp(C\log(q_1T_{l})\log\log x+O(\sqrt{\log q}(\log\log q)^{-1})) \]
\[
\ll\sqrt{x}\exp(\sqrt{\log x}(\log\log x)^{c+2}).
\]
In the case $q>\exp(\sqrt{\log x})$, our choice of segment 1 contributes a constant only.

By Proposition~\ref{completebd}, a segment from group 2 gives a contribution of
\[
\frac1{2\pi}\left|\int_{1/2+V_n^{\chi}/\log x+in}^{1/2+V_n^{\chi}/\log x+i(n+1)}\frac{x^s}{sL(s,\chi)l_d(s,\chi)}ds\right|
\]
\[
\le\frac1{2\pi n}\exp(c\sqrt{\log d}(\log\log d)^{-1})\max_{z\in[1/2+V_n^{\chi}/\log x+in,1/2+V_n^{\chi}/\log x+i(n+1)]}|x^zL(z,\chi)^{-1}|
\]
\[
\le\frac1n\sqrt{x}\exp\left(V_n^{\chi}\log\left(\frac{\log x}{\log(qn)}\right)+2(1+\del)V_n^{\chi}\log\log V_n^{\chi}+O\left(\frac{V_n^{\chi}}{\del^2}+\frac{\sqrt{\log x}}{\log\log x}\right)\right).
\]
From segment 3 we get, by Proposition~\ref{genericbd},
\[
\frac1{2\pi}\left|\int_{\frac12+\frac1{\log x}+iT_{l}}^{\frac12+\frac{V_{l}^{\chi}}{\log x}+iT_{l}}\frac{x^s}{sL(s,\chi)l_d(s,\chi)}ds\right|\le\sqrt xT_{l}^3\exp(c\sqrt{\log d}(\log\log d)^{-1}).
\]
A segment from group 4, by Proposition~\ref{completebd}, gives
\[
\frac1{2\pi}\left|\int_{1/2+V_n^{\chi}/\log x+in}^{1/2+V_n^{\chi}/\log x+i(n+1)}\frac{x^s}{sL(s,\chi)l_d(s,\chi)}ds\right|
\]
\[
\le\frac1n\sqrt{x}\exp\left(V_n^{\chi}\log\left(\frac{\log x}{\log(qn)}\right)+2(1+\del)V_n^{\chi}\log\log V_n^{\chi}+O\left(\frac{V_n^{\chi}}{\del^2}+\frac{\sqrt{\log x}}{\log\log x}\right)\right)
\]
\[
+\frac1{n+1}\sqrt{x}\exp\left(V_{n+1}^{\chi}\log\left(\frac{\log x}{\log(q(n+1))}\right)+2(1+\del)V_{n+1}^{\chi}\log\log V_{n+1}^{\chi}+O\left(\frac{V_{n+1}^{\chi}}{\del^2}+\frac{\sqrt{\log x}}{\log\log x}\right)\right).
\]
By Proposition~\ref{largesglogbound}, for segment 5, we have
\[
\frac1{2\pi}\left|\int_{\frac12+\frac{V_{T_K-1}^{\chi}}{\log x}+iT_K}^{\frac12+\frac{1}{\log x}+iT_K}\frac{x^s}{sL(s,\chi)l_d(s,\chi)}ds\right|\ll_{\del}\sqrt{x}\exp(c\sqrt{\log d}(\log\log d)^{-1}).
\]
\end{proof}

Before we proceed to bound $B(x,\chi)$ and complete the proof, we need one final elementary inequality.

\begin{lem}
\label{elem2}
Let $A$ and $C$ be positive numbers such that $A\ge4C^4+1$. Then for all $V>e^C$ we have
\[
AV-V\log V+CV\log\log V\le e^AA^C.
\]
\end{lem}

\begin{proof}
This is Proposition 23 in~\cite{baladero}. We reproduce the proof here.

Let $f(V)=AV-V\log V+CV\log\log V$. We compute that
\[
f'(V)=A-\log V+C\log\log V-1+\frac{C}{\log V}
\]
and
\[
f''(V)=-\frac1V+\frac{C}{V\log V}-\frac{C}{V(\log V)^2}.
\]
We can see from this that $f''(V)<0$ for $V>e^C$, that
\[
f'(e^C)=A-C+C\log C-1+1\ge 4C^4+1-C+C\log C>0,
\]
and that $f'(\infty)=-\infty$. Hence there is a unique $V_0>e^C$ such that $f'(V_0)=0$, and we have
\[
\max_{V\ge e^C}f(V)=f(V_0)=V_0(A-\log V_0+C\log\log V_0)=V_0(1-C/\log V_0)\le V_0.
\]
Let $V_1=e^AA^C$. Then
\[
f'(V_1)=A-(A+C\log A)+C\log(A+C\log A)-1+\frac{C}{A+C\log A}
\]
\[
\le C\log\left(1+\frac{C\log A}A\right)-1+\frac CA \le C\log\left(1+\frac{C}{\sqrt{A}}\right)-1+\frac CA
\]
\[
\le\frac{C^2}{\sqrt{A}}-1+\frac CA\le0.
\]
Therefore $AV-V\log V+CV\log\log V\le V_0\le V_1=e^AA^C$, as desired.
\end{proof}

Now we apply our count of $V$-atypical ordinates from Section~\ref{pointcount} to bound $B(x,\chi)$.

\begin{lem}
We have
\[
\sum_{\chi}B(x,\chi)\ll_{\eps}\phi(q)\exp\left((\log x)^{1/2}(\log\log x)^{\max\{3+\eps,5-c+\eps\}}\right).
\]
\end{lem}

\begin{proof}
Let
\[
B(T,x,\chi)=\sum_{T\le n<2T}\frac1n\exp\left(V_n^{\chi}\log\left(\frac{\log x}{\log(qn)}\right)+2(1+2\del)V_n^{\chi}\log\log V_n^{\chi}\right).
\]
We first rearrange the sum to put together the terms corresponding to the same value of $V_n^{\chi}$. Let $M(V,T,\chi)=\{n\in\nats\mid T\le n<2T,V_n^{\chi}=V\}$. Then
\[
\sum_{\chi}B(T,x,\chi) =\sum_{\chi}\sum_{a(T,q)\le V\le b(T,q)}\sum_{\substack{T\le n<2T\\V_n^{\chi}=V}}\frac1n\exp\left(V\log\left(\frac{\log x}{\log(qn)}\right)+2(1+2\del)V\log\log V\right)
\]
\[
\le\frac1T\sum_{a(T,q)\le V\le b(T,q)}\exp\left(V\log\left(\frac{\log x}{\log(qT)}\right)+2(1+2\del)V\log\log V\right)\sum_{\chi}|M(V,T,\chi)|.
\]
For $V\le2a(T,q)+1$ we note that trivially $|M(V,T,\chi)|\le T$, and therefore
\[
\frac1T\sum_{a(T,q)\le V\le2a(T,q)+1}\exp\left(V\log\left(\frac{\log x}{\log(qT)}\right)+2(1+2\del)V\log\log V\right)\sum_{\chi}|M(V,T,\chi)|
\]
\[
\ll\phi(q)\sum_{a(T,q)\le V\le2a(T,q)+1}\exp\left(C\sqrt{\log q}(\log\log qT)^2(\log\log x)\right)
\]
\[
\ll_{\eps}\phi(q)\exp\left(\sqrt{\log x}(\log\log x)^{3+\eps}\right).
\]
For larger $V$, 
note that since $V_n^{\chi}$ is the smallest integer such that all $t\in[n,n+1]$ are $(V_n^{\chi},\del,\chi)$-typical of order $T$, there is some $t_n^{\chi}\in[n,n+1]$ which is $(V_n^{\chi}-1,\del,\chi)$-atypical of order $T$. Then for $i\in\{0,1\}$, the set $N_i(V,T)=\cup_{\chi}\{t_n^{\chi}\mid n\in M(V,T,\chi), n\equiv i\pmod 2\}$ is a set of well-separated ordinates satisfying the hypotheses of Proposition~\ref{vtypicalbd}, so we conclude that
\[
|N_i(V,T)|\ll_{\del}\phi(q)T\exp\left(-V\log\left(\frac{V}{\log\log(qT)}\right)+(2+\del)V\log\log V\right).
\]
But $\sum_{\chi}|M(V,T,\chi)|=|N_0(V,T)|+|N_1(V,T)|$. Thus we also have
\[
\sum_{\chi}|M(V,T,\chi)|\ll_{\del}\phi(q)T\exp\left(-V\log\left(\frac{V}{\log\log(qT)}\right)+(2+\del)V\log\log V\right).
\]
Plugging this in gives
\[
\sum_{\chi}B(T,x,\chi)\ll_{\del,\eps}O(\phi(q)\exp(\sqrt{\log x}(\log\log x)^{3+\eps}))
\]
\[
+\phi(q)\sum_{2a(T)+1\le V\le b(T)}\exp\left(V\log\left(\frac{\log x\log\log(qT)}{\log(qT)}\right)-V\log V+(4+5\del)V\log\log V\right).
\]
We apply Lemma~\ref{elem2} with $A=\log\left(\frac{\log x\log\log(qT)}{\log(qT)}\right)$ and $C=4+5\del$, noting that $A\ge\log\left(\frac{\log\log(qT)}2\right)\ge4C^4+1$ and $V>e^C$ for $qT$ sufficiently large. Then we get
\[
V\log\left(\frac{\log x\log\log(qT)}{\log(qT)}\right)-V\log V+(4+5\del)V\log\log V
\]
\[
\le\log x\frac{\log\log(qT)}{\log(qT)}\left(\log\left(\log x\frac{\log\log(qT)}{\log(qT)}\right)\right)^{4+5\del}
\]
By our choice of $l$ we have $qT\ge\exp((\log x)^{1/2}(\log\log x)^c)$, so this is
\[
\ll\log x\frac{(\log\log x)^{1-c}}{(\log x)^{1/2}}\left(\log\left(\log x\frac{(\log\log x)^{1-c}}{(\log x)^{1/2}}\right)\right)^{4+5\del}\ll(\log x)^{1/2}(\log\log x)^{5-c+5\del}.
\]
We plug back in to $\sum B(T,x,\chi)$ to get
\[
\sum_{\chi}B(T,x,\chi)
\]
\[
\ll_{\del,\eps}O(\phi(q)\exp((\log x)^{1/2}(\log\log x)^{3+\eps}))
+\phi(q)\frac{\log(qT)}{\log\log(qT)}\exp((\log x)^{1/2}(\log\log x)^{5-c+5\del})
\]
and hence
\[
\sum_{\chi}B(x,\chi)=\sum_{\substack{T=T_k\\ l\le k\le K}}\sum_{\chi}B(T,x,\chi)\exp\left(D\frac{\sqrt{\log x}}{\log\log x}\right)
\]
\[
\ll_{\del,\eps}\phi(q)\exp\left((\log x)^{1/2}\left((\log\log x)^{3+\eps}+(\log\log x)^{5-c+5\del}+\frac{D}{\log\log x}\right)\right).
\]
Setting $\del=\eps/5$ gives the bound we want.
\end{proof}

Our main theorem now follows immediately.

\begin{proof}[Proof of Theorem~\ref{main}]
From the preliminaries we have
\[
|M(x;q,a)| \le\frac1{\phi(q)}\sum_{\chi}|A(x/d,\chi)|+O(\log(x/d))
\]
\[
\ll_{\eps}\sqrt{x/d}\exp\left((\log(x/d))^{1/2}(\log\log(x/d))^{c+1+\eps}\right)
\]
\[
+\sqrt{x/d}\exp\left((\log(x/d))^{1/2}\left((\log\log(x/d))^{3+\eps}+(\log\log(x/d))^{5-c+\eps}\right)\right)
\]
\[
\ll_{\eps}\sqrt{x/d}\exp\left((\log(x/d))^{1/2}(\log\log(x/d))^{3+\eps}\right)
\]
after setting $c=2$.
\end{proof}

\end{document}